\documentclass[12pt]{amsart}

 \usepackage{amsfonts,graphics,amsmath,amsthm,amsfonts,amscd,amssymb,amsmath,latexsym,multicol}
\usepackage{epsfig,url}
\usepackage{flafter}

\makeatletter

\def\jobis#1{FF\fi
  \def\predicate{#1}%
  \edef\predicate{\expandafter\strip@prefix\meaning\predicate}%
  \edef\job{\jobname}%
  \ifx\job\predicate
}

\makeatother

\if\jobis{proposal}%
\else
\fi

 \usepackage[matrix, arrow]{xy}

\DeclareMathOperator{\Supp}{Supp}

 \newcommand{\N}{\mathbb N}
 
 \newcommand{\Q}{\mathbb Q}
 \newcommand{\R}{\mathbb R}
 
 \newcommand{\bir}{\dashrightarrow}

 \numberwithin{equation}{subsection}
 \numberwithin{footnote}{subsection}

 \newtheorem{cor}[subsection]{Corollary}

 \newtheorem{thm}[subsection]{Theorem}

\newtheorem{claim}[subsection]{Claim}

{
 \newtheorem{defn}[subsection]{Definition}

}

 \newcommand{\ke}[1]{$\acute{\mbox{e}}$}
 \newcommand{\ku}[1]{$\acute{\mbox{u}$}}
 \newcommand{\kl}[1]{$\acute{\mbox{l}}$}
 \newcommand{\kh}[1]{$\acute{\mbox{h}}$}
 \newcommand{\kr}[1]{$\acute{\mbox{r}}$}
 \newcommand{\kx}[1]{$\acute{\mbox{x}}$}
 \newcommand{\ki}[1]{${\^\i}$}

\title{On termination of log flips in dimension four}
\author{Caucher Birkar}
\date{\today}
\begin{document}
\maketitle

\begin{abstract}
We prove the termination of $4$-fold log flips for klt pairs of Kodaira dimension $\kappa\ge 2$.
\end{abstract}


\section{Introduction}

We work over an algebraically closed field $k$ of characteristic zero. See section 2 for notations.

Let $(X/Z,B)$ be a klt pair. We say that $(X/Z,B)$ has \emph{log canonical dimension} $\lambda(X/Z,B)$ 
if $(X/Z,B)$ has a log minimal model $(\overline{X}/Z,\overline{B})$ such that $K_{\overline{X}}+\overline{B}$ is semi-ample/$Z$ and its log canonical model 
has dimension $\lambda(X/Z,B)$. If $(X/Z,B)$ does not have such a log minimal model, we let $\lambda(X/Z,B)=-\infty$. On the other hand, when $B$ is a $\Q$-divisor, we define the Kodaira dimension 
$$
\kappa(X/Z,B):=\kappa((K_X+B)|_F)+\dim Z
$$ 
where $F$ is the generic fibre of the surjective morphism $X\to Z$ and $\kappa((K_X+B)|_F)$ is the usual Kodaira dimension of the divisor $(K_X+B)|_F$ on $F$. If the minimal model conjecture and the abundance conjecture hold, then of course $\lambda(X/Z,B)=\kappa(X/Z,B)$. 
 
Though the minimal model conjecture is settled in dimension $4$ [\ref{B2}][\ref{B3}][\ref{Sh}] but the termination problem is known only in some special cases. The following theorem is another step in this direction. 

\begin{thm}\label{t-main}
Let $(X/Z,B)$ be a klt pair of dimension $4$ such that $\lambda(X/Z,B)\ge 2$. Then, any sequence of log flips 
starting with $(X/Z,B)$ terminates.
\end{thm}

\begin{cor}
Let $(X/Z,B)$ be a klt pair of dimension $4$ such that $B$ is a $\Q$-divisor and $\kappa(X/Z,B)\ge 2$. 
Then, any sequence of log flips starting with $(X/Z,B)$ terminates.
\end{cor}
\begin{proof}
Let $F$ be the generic fibre of $X\to Z$. From the assumptions it is evident that $\kappa((K_X+B)|_F)\ge 0$. Thus, $(X/Z,B)$ has a log minimal model by [\ref{B2}] or [\ref{B3}] or [\ref{Sh}]. Now apply [\ref{Fujino}, Theorem 2.8] to derive $\lambda(X/Z,B)\ge 2$ and use Theorem \ref{t-main}.
\end{proof}

It should not be difficult to generalise these results to the lc case but for simplicity we only treat the klt case. Finally, we remark that both the ACC conjecture for log canonical thresholds and the ACC conjecture for minimal log discrepancies 
 in dimension $4$ imply stronger versions of Theorem \ref{t-main} [\ref{B}][\ref{Sh-mld's}].


\section{Preliminaries}

We work over an algebraically closed field $k$ of characteristic zero.

\begin{defn}
A pair $(X/Z,B)$ consists of normal quasi-projective varieties $X,Z$ over $k$, an $\R$-divisor $B$ on $X$ with
coefficients in $[0,1]$ such that $K_X+B$ is $\mathbb{R}$-Cartier, and a surjective projective morphism $X\to Z$. For a prime divisor $D$ on some birational model of $X$ with a nonempty centre on $X$, $C_XD$ denotes its centre and $a(D,X,B)$
the log discrepancy.
\end{defn}

\begin{defn}\label{d-model}
Let $(X/Z,B)$ be a klt pair. A klt pair $(\overline{X}/Z,\overline{B})$ is called a log minimal model of $(X/Z,B)$ 
if there is a birational map $\phi\colon X\bir \overline{X}/Z$, such that $\phi^{-1}$ does not contract divisors, 
$\overline{B}=\phi_*B$, $\overline{X}$ is $\Q$-factorial, $K_{\overline{X}}+\overline{B}$ is nef/$Z$, and finally 
 for any prime divisor $D$ on $X$ which is exceptional$/\overline{X}$, we have
$$
a(D,X,B)<a(D,\overline{X},\overline{B})
$$

If $K_{\overline{X}}+\overline{B}$ is semi-ample$/Z$, then there is a contraction $\psi\colon \overline{X}\to S/Z$ and 
an ample/$Z$ $\R$-divisor $H$ on the normal variety $S$ such that $K_{\overline{X}}+\overline{B}\sim_{\R} \psi^*H/Z$. 
We call $S$ the log canonical model of $(X/Z,B)$. 
\end{defn}

Note that this definition of log minimal models is equivalent to that of [\ref{B2}, \S 2] and [\ref{B3}, \S 2] in the klt case.

By a log flip we mean the flip of an extremal flipping contraction  for a log divisor [\ref{B}, \S 2]. A log flip is called of type 
$(a,b)$ if the flipping locus has dimension $a$ and the flipped locus has dimension $b$. It is 
well-known that a log flip in dimension $4$ is of type $(1,2), (2,1)$ or $(2,2)$. 

 A sequence of log flips/$Z$ starting with $(X/Z,B)$ is a sequence  $X_i\dasharrow X_{i+1}/Z_i$ where 
$i\in \N$, $(X_1/Z,B_1)=(X/Z,B)$ and  $X_i\to Z_i \leftarrow X_{i+1}$ is a $K_{X_i}+B_i$-flip$/Z$ where  $B_i$ is the 
birational transform of $B$.


\section{Proof of the theorem}

\begin{proof}
\emph{Step 1.} 
Let $X_i\bir X_{i+1}/Z_i$ be a sequence of log flips starting with a klt pair $(X_1/Z,B_1)=(X/Z,B)$ 
of dimension $4$ such that $\lambda(X/Z,B)\ge 2$.  
We can assume that $X$ and so all the $X_i$ are $\Q$-factorial by lifting the sequence  using a $\Q$-factorialisation of $(X/Z,B)$ ( see [\ref{B}, Construction 3.1] or [\ref{B3}, Remark 2.4]).\\ 

\emph{Step 2.} 
Since $\lambda(X/Z,B)\ge 2$, there is an $\R$-divisor $M\ge 0$ such that $K_X+B\sim_\R M/Z$. By adding 
a small multiple of $M$ to $B$, we can assume that $B-\epsilon M\ge 0$ for some $\epsilon >0$. 
 By [\ref{AHK}, Theorem 2.15], for $i\gg 0$, the support of $B_i$ does not contain any $2$-dimensional 
 component of the flipping locus or the flipped locus of the log flip $X_i\bir X_{i+1}/Z_i$. But the support of 
$B_i$ contains the entire flipping locus because the support of $M_i$, the birational transform of $M$, does so 
and because $\Supp M_i\subseteq \Supp B_i$. 
So, by truncating the sequence we can assume that for any $i$ 
the support of $B_i$ does not contain any $2$-dimensional  component of the flipping locus or the 
flipped locus of $X_i\bir X_{i+1}/Z_i$ 
hence all the log flips are of type $(1,2)$. For $i>1$, let $V_{i}\subset X_{i}$ be a $2$-dimensional component of the flipped locus of 
 $X_{i-1}\bir X_{i}/Z_{i-1}$.\\  

\emph{Step 3.} 
By [\ref{B3}] or [\ref{Sh}], $(X_i/Z,B_i)$ has a log minimal model $(\overline{X}_i/Z,\overline{B}_i)$,  
equipped with a birational map $\phi_i\colon X_i\bir \overline{X}_i$, which 
 is obtained from $(X_i/Z,B_i)$ by a sequence of log flips and divisorial contractions. More precisely, 
the log minimal model is obtained by running the LMMP with scaling. In particular, 
the restriction of all the induced birational maps $X_1\bir X_i$ and $X_1\bir \overline{X}_i$ to $U=X_1-\Supp M_1$
is an isomorphism onto the image of $U$. This follows from the fact that any extremal ray contracted 
in the process of $X_1\bir X_i$ or $X_i\bir \overline{X}_i$ should be inside the support of the birational transform of $M_1$. 
Moreover, since $V_i$ is not inside $\Supp M_i$, $V_i$ is not inside the exceptional locus of $\phi_i\colon X_i\bir \overline{X}_i$ hence it 
has a birational transform on $\overline{X}_i$ which we denote by $\overline{V}_i$. 
\\ 

\emph{Step 4.} 
Since all the $(\overline{X}_i/Z,\overline{B}_i)$ are log minimal models of $(X/Z,B)$, and since $\lambda(X/Z,B)\ge 2$, each $K_{\overline{X}_i}+\overline{B}_i$ is semi-ample$/Z$ and it has a log canonical model $S/Z$ of dimension $\ge 2$ 
and a contraction $\psi_i\colon \overline{X}_i\to S$ as in Definition \ref{d-model}. In fact, $S$ does not deepened on $i$ and there is an ample$/Z$ $\R$-divisor $H$, 
independent of $i$,  such that $K_{\overline{X}_i}+\overline{B}_i\sim_\R \psi_i^*H/Z$.\\ 

\emph{Step 5.} 
By the claim below, one can choose the ample$/Z$ $\R$-divisor $H\ge 0$ on $S$ such that $W_i\subseteq \Supp H$ for infinitely 
many $i$ where  $W_i=\psi_i(\overline{V}_i)$. Let $\overline{N}_i=\psi_i^*H$. Then,  there is an effective $\R$-divisor $N_1\sim_\R K_{X_1}+B_1/Z$ on $X_1$ 
such that $\overline{N}_i$ is the birational transform of $N_1$ for every $i$. Let $N_i$ on $X_i$ be the birational transform of $N_1$. Now by  applying [\ref{AHK}] to $(X_1/Z,B_1+\delta N_1)$,  
for some small $\delta>0$, we deduce that $\Supp {N}_i$ does not contain ${V}_i$ for $i\gg 0$. This in turn implies that $\Supp \overline{N}_i$ does not contain $\overline{V}_i$ for 
$i\gg 0$. This is a contradiction because by construction $\overline{V}_i\subseteq \Supp \overline{N}_i$ for infinitely many $i$.   

From now on we deal with stating and proving the claim below. By truncating the sequence we may assume that $X_i$ is terminal at the generic point of $V_i$, i.e. $X_i$ is smooth at the generic point of $V_i$ [\ref{AHK}, Lemma 2.14]. Let $E_i$ be an 
exceptional$/X_i$ prime divisor such that $a(E_i,X_i,B_i)=2$ 
and such that its centre $C_{X_i}E_i=V_i$. Such $E_i$ can be obtained by blow up along $V_i$. 

\begin{claim}\label{claim}
Under the above notation and assumptions (in Step 1-5), we have the following properties:\\ 
(1) $C_{\overline{X}_j}E_i\neq C_{\overline{X}_j}E_j$ if $i<j$,\\
(2) $C_{\overline{X}_1}E_i \nsubseteq\Supp \overline{M}_1$ for any $i$: we use $\overline{M}_1$ to denote the birational transform of $M_1$ (similarly for every $i$),\\
(3) the closure of $\Theta:=\bigcup C_{\overline{X}_1}E_i$ in $\overline{X}_1$ is of dimension $\le 2$,\\
(4) the ample/$Z$ $\R$-divisor $H\ge 0$ can be chosen so that $W_i\subseteq \Supp H$ for infinitely many $i$ 
 where  $W_i=\psi_i(\overline{V}_i)$.
\end{claim}
\emph{Proof.} 
(1) This follows from the fact that $C_{{X}_j}E_i\neq C_{{X}_j}E_j$ on $X_j$ if $i<j$. 

(2)  By construction, $C_{\overline{X}_i}E_i\nsubseteq \Supp \overline{M}_i$ because $C_{{X}_i}E_i\nsubseteq \Supp {M}_i$. On the other hand, 
$$
a(E_i,\overline{X}_1,\overline{B}_1+\tau \overline{M}_1)= a(E_i,\overline{X}_i,\overline{B}_i+\tau \overline{M}_i)
$$ 
for any $0\le \tau \ll 1$ and this is possible only if $C_{\overline{X}_1}E_i\nsubseteq \Supp \overline{M}_1$. 

(3)  This is obvious since $\phi^{-1}_1\colon \overline{X}_1\bir X_1$ is birational and it does not contract any divisors. 

(4) If $\dim S\ge 3$, this follows from (3). So, we assume that $\dim S=2$. 
By construction, $a(E_i,\overline{X}_1,\overline{B}_1)=a(E_i,\overline{X}_i,\overline{B}_i)=2$ for all $i$. This could happen only if 
on a terminal crepant model $(\overline{X}_1^t/Z,\overline{B}_1^t)$ of $(\overline{X}_1/Z,\overline{B}_1)$, almost 
all $C_{\overline{X}_1^t}E_i$ are surfaces [\ref{AHK}, Lemma 1.5]. Though [\ref{AHK}, Lemma 1.5] deals only with log discrepancy $<2$ but we still can use it by taking a 
divisor $L\ge 0$ containing $\Theta$ and applying the lemma to $(\overline{X}_1/Z,\overline{B}_1+\tau L)$ for some small $\tau>0$.

By [\ref{Kawamata}], the induced birational map $\overline{X}_1\bir \overline{X}_{i}$ is 
decomposed into a sequence $Y_j\bir Y_{j+1}/T_j$ of flops$/S$ in which $Y_1=\overline{X}_1$ and $Y_m=\overline{X}_{i}$. 
 When $i\gg 0$, the centre of $E_i$ should be in the flopping locus somewhere in the sequence because 
of (1) and (3). Suppose that $C_{Y_k}E_i$ is in the flopping locus  
of $Y_k\bir Y_{k+1}/T_k$, that is, the exceptional locus of $Y_k\to T_k$. In particular, $\dim C_{T_k}E_i\le 1$ and since $T_k/S$ we deduce that $\dim W_i\le 1$.\\ 

Now let $\Sigma$ be a minimal (in the sense of inclusion) subvariety of $\Theta$ which contains $C_{\overline{X}_1}E_i$ 
for infinitely many $i$. Let $\sigma=\{i\mid C_{\overline{X}_1}E_i\subseteq \Sigma \}$. Clearly, $\dim \Sigma\le 2$. 
If $\dim \psi_1(\Sigma)\le 1$, we are done, so we may assume that $\Sigma$ is a surface and the restricted  
map $\psi_1\colon \Sigma \to S$ is generically finite. Since $\dim W_i\le 1$ for almost all $i$, $C_{\overline{X}_1}E_i\neq \Sigma$ 
for almost all $i\in \sigma$ which means that $\dim C_{\overline{X}_1}E_i\le 1$ for almost all $i\in \sigma$. This in turn 
implies that $C_{\overline{X}_1^t}E_i$ is inside a prime exceptional divisor $G$ of the morphism $\overline{X}_1^t \to \overline{X}_1$ 
for all $i\in \sigma'$ where $\sigma'\subseteq \sigma$ is some infinite subset. Since $\Sigma$ was chosen to be minimal, $G$ is mapped onto 
$\Sigma$, and so $\dim C_{\overline{X}_1}E_i=1$ for almost all 
$i\in \sigma'$. Since $\Sigma\to S$ is generically finite, $\dim W_i=1$ for almost all $i\in \sigma'$. In addition, we can assume that $\{W_i\}_{i\in \sigma'}$ is not a finite set otherwise we are done.

Choose a large $i\in \sigma'$. So, $\dim C_{\overline{X}_1}E_i=\dim W_i=1$ and as mentioned earlier, the centre of $E_i$ 
is in the exceptional locus of some $Y_k\to T_k$. So, there is a surface $\Gamma_i\subseteq Y_k$ which is in the exceptional 
locus of $Y_k\to T_k$ and such that it is mapped onto $C_{T_k}E_i$ and also onto $W_i$. Now let 
$l\le k$ be the smallest number such that there is a surface $\Gamma'_i$ which is in the exceptional  
locus of $Y_l\to T_l$ and such that it is mapped onto $W_i$. The surface $\Gamma'_i$ is the birational 
transform of a surface $\Gamma''_i$ on $\overline{X}_1$ by the minimality of $l$.

Note that since $\overline{M}_1$ is effective and numerically zero over $S$, it is not horizontal over $S$, that is, it is not mapped onto $S$. On the other hand, if  $\Gamma''_i\subseteq \Supp \overline{M}_1$, then $W_i$ is contained in the image of $\overline{M}_1$ on $S$ hence we may assume that $\Gamma''_i\nsubseteq \Supp \overline{M}_1$  
for almost all $i\in \sigma'$. Therefore, either $\{\Gamma_i''\}_{i\in\sigma'}$ is finite, or $\Gamma_i''$ has a  birational transform on $X_1$ so also on $X_i$ and $\overline{X}_i$, for infinitely many $i\in \sigma'$. The former contradicts the assumption that $\{W_i\}_{i\in \sigma'}$ is not finite, and the latter contradicts the fact that $\Gamma_i''$ is flopped in the sequence $Y_j\bir Y_{j+1}/T_j$.
\end{proof}

\section{Some remarks}

In the proof of Theorem \ref{t-main}, if there were only finitely many log minimal models  $\overline{X}_i$, then the proof would work without the restriction  $\lambda(X/Z,B)\ge 2$. However, in general the number of log  minimal models is not finite but it is conjectured that they are finite up to isomorphism, that is, when we forget about the induced birational relations 
$\overline{X}_i\bir \overline{X}_j$. This weak finiteness follows from a conjecture of Morrison, Kawamata and Totaro about Calabi-Yau fibre spaces [\ref{Totaro}, Conjecture 8.1][\ref{Kawamata2}]. Most probably a modification of the proof of Theorem \ref{t-main} works if one has this weak finiteness.


\vspace{2cm}

\flushleft{DPMMS}, Centre for Mathematical Sciences,\\
Cambridge University,\\
Wilberforce Road,\\
Cambridge, CB3 0WB,\\
UK\\
email: c.birkar@dpmms.cam.ac.uk

\end{document}